\documentclass[leqno]{article}
\usepackage{geometry}
\usepackage{graphicx}	
\usepackage[cp1251]{inputenc}
\usepackage[english]{babel}
\usepackage{mathtools}
\usepackage{amsfonts,amssymb,mathrsfs,amscd,amsmath,amsthm}

\usepackage{verbatim}

\usepackage{url}

\usepackage{stmaryrd}

\ifpdf
\usepackage{hyperref}
\else
\usepackage[dvipdfmx]{hyperref}
\fi

\def\thtext#1{
  \catcode`@=11
  \gdef\@thmcountersep{. #1}
  \catcode`@=12
}

\def\threst{
  \catcode`@=11
  \gdef\@thmcountersep{.}
  \catcode`@=12
}

\theoremstyle{plain}

\newtheorem{thm}{Theorem}[section]
\newtheorem{prop}[thm]{Proposition}
\newtheorem{cor}[thm]{Corollary}
\newtheorem{ass}[thm]{Assertion}
\newtheorem{lem}[thm]{Lemma}

\theoremstyle{definition}
\newtheorem{dfn}[thm]{Definition}
\newtheorem{rk}[thm]{Remark}

\newtheorem{examp}[thm]{Example}

 \catcode`@=11
 \def\.{.\spacefactor\@m}
 \catcode`@=12

\def\R{\mathbb R}

\def\e{\varepsilon}
\def\dl{\delta}
\def\D{\Delta}
\def\g{\gamma}
\def\G{\Gamma}

\def\l{\lambda}
\def\r{\rho}
\def\s{\sigma}

\def\v{\varphi}

\def\0{\emptyset}
\def\:{\colon}
\def\<{\langle}
\def\>{\rangle}
\def\[{\llbracket}
\def\]{\rrbracket}

\def\rom#1{\emph{#1}}
\def\({\rom(}
\def\){\rom)}
\def\sm{\setminus}
\def\ss{\subset}

\def\sp{\supset}

\def\x{\times}

\def\bA{{\bar A}}

\def\cut{\operatorname{CUT}}

\def\diam{\operatorname{diam}}

\def\dis{\operatorname{dis}}

\def\NP{\operatorname{NP}}

\def\cC{{\cal C}}

\def\cG{{\cal G}}
\def\cH{{\cal H}}

\def\cM{{\cal M}}

\def\cP{{\cal P}}
\def\cR{{\cal R}}

\def\cU{{\cal U}}

\begin{document}
\title{Embeddings to Rectilinear Space and \\
Gromov--Hausdorff Distances}
\author{A.\,O.~Ivanov, A.\,A.~Tuzhilin}
\date{}
\maketitle

\begin{abstract}
We show that the problem whether a given finite metric space can be embedded into $m$-dimensional rectilinear space can be reformulated in terms of the Gromov--Hausdorff distance between some special finite metric spaces.

{\bf MSC~2020:} 46B85, 51F99, 53C23.

\end{abstract}

\section{Introduction}
\markright{\thesection.~Introduction}
We continue our investigations~\cite{TusMST}, \cite{ITBors}, \cite{ITBorsFom}, \cite{ITBUSp} on connections between the geometry of Gromov--Hausdorff distance and Discrete Geometry problems such as calculation of edges weights of a minimal spanning tree, Borsuk conjecture, chromatic number and clique cover number calculation for a simple graph, etc. In this paper we give an answer to a question whether a given finite metric space $(X,d)$ can be embedded into $m$-dimensional rectilinear space in terms of the Gromov--Hausdorff distance.

\subsubsection*{Hyperspaces}
For subsets $A$ and $B$ of a fixed metric space $X$, a natural distance function $d_H$ was defined by F.~Hausdorff~\cite{Hausdorff} as the infimum of positive numbers $r$ such that $A$ is contained in the $r$-neighborhood of $B$, and vice-versa. It is well-known that this function, referred as the \emph{Hausdorff distance}, is a metric on the family of all closed bounded subsets of the metric space $X$, see for example~\cite{BurBurIva}. The Hausdorff distance was generalized to the case of two metric spaces $X$ and $Y$ by D.~Edwards~\cite{Edwards} and independently by M.~Gromov~\cite{Gromov}. The \emph{Gromov--Hausdorff distance\/} between the spaces $X$ and $Y$ equals the infimum of the values $d_H\big(\v(X),\psi(Y)\big)$ over all possible isometrical embeddings $\v\:X\to Z$ and $\psi\:Y\to Z$ into all possible metric spaces $Z$.  It is well-known that this distance function is a metric on the family of isometry classes of non-empty compact metric spaces. The corresponding hyperspace is usually denoted by $\cM$ and is referred as the \emph{Gromov--Hausdorff space\/} (see~\cite{BurBurIva}).

The geometry of the Gromov--Hausdorff space is rather tricky and is intensively investigated by many authors, see, for example, a review in~\cite{BurBurIva}. The technique of optimal correspondences permitted to prove that the space $\cM$ is geodesic~\cite{IvaNikolaevaTuz, ChowdhuryMemoli, IvaIlTuzRealiz} and to describe some local and all global isometries of $\cM$, see~\cite{IvaTuzLocalStrIsom, ITautoGH}. Since finite metric spaces form an everywhere dense subset of $\cM$, the distances to such spaces and between such spaces play an important role in the research of geometry of $\cM$. Important classes of such spaces are formed by the ones all whose non-zero distances are the same (so-called \emph{single-distance spaces\/} or \emph{simplexes\/}) and by the spaces whose non-zero distances take only two different values (so-called \emph{two-distance spaces\/}). The authors, together with S.~Iliadis and D.~Grigor'ev, see~\cite{IvaTuzSimpDist8}, \cite{GrigIvaTuzSimpDist}, calculated distances from a bounded metric space to a simplex, and, as a particular case, the distances between any simplex and any $2$-distance space, see~\cite{IvaTuzSimpTwoDist}. It turns out that the Gromov--Hausdorff distance from a metric space $X$ to a simplex ``feels'' somehow a geometry of partitions of the space $X$. The latter explains some relations between the Gromov--Hausdorff distance and Discrete Geometry Problems.

\subsubsection*{Discrete Geometry Problems}
In 1933, K.~Borsuk stated the following question: How many parts one needs to partition an arbitrary subset of the Euclidean space into, to obtain pieces of smaller diameters? He made the following famous conjecture: Any bounded non-single-point subset of $\R^n$ can be partitioned into at most $n+1$ subsets, each of which has smaller diameter than the initial subset. K.~Borsuk himself proved it for $n=2$ and for a ball in $3$-dimensional space, \cite{BorCong} and~\cite{Borsuk}. For $n=3$ the conjecture was proved by J.~Perkal (1947) and by H.\,G.~Eggleston (1955), then for convex subsets with smooth boundaries in 1946 by H.~Hadwiger~\cite{Hadw, Hadw2}, and for central symmetric bodies by A.\,S.~Riesling (1971). However, in 1993 the conjecture was disproved in general case  by J.~Kahn, and G.~Kalai, see~\cite{KahnKalai}. They constructed a counterexample in dimension $n=1325$, and also proved that the conjecture is not valid for all $n>2014$.  This estimate was consistently improved by Raigorodskii, $n\ge561$, Hinrichs and Richter, $n\ge298$, Bondarenko, $n\ge 65$, and Jenrich, $n\ge 64$, see details in a review~\cite{Raig}. Notice that all the examples are finite subsets of the corresponding spaces, and the best known results of Bondarenko~\cite{Bond} and Jenrich~\cite{Jenr} are the $2$-distance subsets of the unit sphere.

On the other hand, Lusternik and Schnirelmann~\cite{LustSch}, and a bit later independently Borsuk~\cite{BorCong} and~\cite{Borsuk}, see also~\cite{Zieg}, have shown that the standard sphere and the standard ball in $\R^n$, $n\ge2$, cannot be partitioned into  $m\le n$ subsets having smaller diameters. Thus, the least possible number of parts of smaller diameter, necessary to partition the sphere and the ball in $\R^n$ equals $n+1$.

In paper~\cite{ITBorsFom} we stated a generalized Borsuk problem, passing to an arbitrary bounded metric space  $X$ and its partitions of an arbitrary cardinality $m$ (not necessary finite), and gave a criterion solving the Borsuk problem in terms of the Gromov--Hausdorff distance. Namely, it is shown that to verify the existence of an $m$-partition into subsets of smaller diameter, it suffices to calculate the Gromov--Hausdorff distance from the space $X$ to a simplex having the cardinality $m$ and a smaller diameter than $X$. As a corollary, a solution to the Borsuk problem for a $2$-distance space $X$ with distances $a<b$ is obtained in terms of the clique cover number of the simple graph $G$ with vertex set $X$, whose vertices $x$ and $y$ are connected by an edge iff $|xy|=a$.

Recall that a \emph{clique cover\/}  of a given simple graph is a cover of the vertex set of the graph by subsets such that every their two vertices are adjacent. Each such subset is called a \emph{clique\/} and is a vertex set of a complete subgraph that is also referred as a  \emph{clique}. The minimum $k$ for which a $k$-element clique cover exists is called the \emph{clique cover number\/} of the given graph. Further, a \emph{graph coloring\/} is an assignment of labels (that are traditionally called ``colors'') to vertices of a graph in such a way that no two adjacent vertices are of the same color. The smallest number of colors necessary to color a graph is called its \emph{chromatic number}. It is not difficult to see  that the clique cover can be considered as a graph coloring of the complement  graph, hence the clique cover number of a graph equals to the chromatic number of the complement one. Calculation and estimation of these numbers are very hard combinatorial problems, see a review in~\cite{Lewis}. In~\cite{ITBorsFom} we calculated the clique cover number of a simple graph and  the chromatic number of a simple graph in terms of the Gromov--Hausdorff distance from an appropriate simplex to the $2$-distance spaces constructed by the graph.

\subsubsection*{Isometrical Embeddings and Main Result}
Another classical problem we deal with is the isometric embedding problem, which consists of asking whether a given metric space can be isometrically embedded into some prescribed ambient space. The most studied case is when the ambient space is Euclidean. This investigations were initiated by Cayley in~1841, see~\cite{Cayley}, and later continued by Menger~\cite{Menger}, Sch\"onberg~\cite{Schon}, Blumenthal~\cite{Blum}, and other specialists. One of Sch\"oenberg’s well-known results asserts that a metric space $(X,d)$ is isometrically embeddable in Hilbert space if and only if the squared distance $d^2$ satisfies a list of linear inequalities. These results were later extended in the context of the Banach $L_p$- and $\ell_p$-spaces. Of particular importance for our purpose is a result by Bretagnolle, Dacunha Castelle and Krivine~\cite{BDCK} stating that $(X, d)$ is isometrically $L_p$-embeddable if and only if the same holds for every finite subspace of $(X, d)$. The same is true for the $\ell_p$-embeddability, see~\cite{DezaLaurent}.

Embeddability into other metric spaces is also interesting and investigated intensively. In 1925, P.~Urysohn~\cite{Ury} found a separable complete metric space $\cU$ containing an arbitrary finite metric space in such ``symmetric'' way that any isometry between finite subsets of $\cU$ can be extended to isometry of the whole $\cU$. Now $\cU$ is known as the universal Urysohn space. In 1935, K.~Kuratowski~\cite{Kurat} constructed an isometric embedding of a metric space $X$ into the Banach space $C_b(X)$ of bounded continuous functions on $X$ endowed with $\sup$-metric. In the case of separable metric spaces the same constriction can be used to get an embedding into the Banach space $\ell_\infty(X)$ of all bounded functions on $X$ with the $\sup$-metric. Another famous result that should be mentioned is J.~Nash Theorem~\cite{Nash} that guaranties embeddability of any Riemannian manifold into an appropriate Euclidean space. Recently the authors in collaboration with S.~Iliadis, see~\cite{IvaTuzLocalStrIsom} and~\cite{ITEmb}, constructed  embeddings of an arbitrary finite metric space in the Gromov--Hausdorff space $\cM$ and of any bounded metric space into the Gromov--Hausdorff metric class (see details in~\cite{ITEmb}).

In this paper we are interested in the case of isometrical embeddings in a rectilinear space. Due to just mentioned results, it suffices to consider finite metric spaces $(X,d)$, and embeddability of a finite pseudometric space into an $\ell_1$- as into an $L_1$-space is equivalent to the distance $d$ belonging to the so-called cut cone, see definition below. This important result was found firstly by Assouad, see~\cite{Assouad} and~\cite{AssouadDeza}. On one hand, the cut cone is defined by a finite system of  linear inequalities, but in contrast with Euclidean case, not all inequalities from this system are known, in general. So, other approaches have great importance. In 1998 H.-J.~Bandelt, V.~Chepoi, and M.~Laurent~\cite{BCL} found out that the question of embeddability of a finite space $(X,d)$ into $m$-dimensional rectilinear space $\R^m_1$ can be reformulated in terms of colorability of certain hypergraph associated with the space $(X,d)$.

We reduce the colorability of the hypergraph to colorability of several simple graphs and then, using our results on the relations between coloring number of a simple graph and Gromov--Hausdorff distance described above, reduce the question about embeddability of a finite space $(X,d)$ into $m$-dimensional rectilinear space to calculation of the Gromov--Hausdorff distance from special $2$-distant spaces associated with $(X,d)$ to one-distant spaces, see Main Theorem.

\section{Preliminaries}
\markright{\thesection.~Preliminaries}

Let $X$ be an arbitrary nonempty set. Recall that a function on $d\:X\x X\to\R$ is called a \emph{metric\/} if it is non-negative, non-degenerate, symmetric, and satisfies the triangle inequality. A set with a metric is called a \emph{metric space}.  If such a function $d$ is permitted to take infinite values, then we call $d$ \emph{a generalized metric}. If we omit the non-degeneracy condition, i.e., permit $d(x,y)=0$ for some distinct $x$ and $y$, then we change the term ``metric'' to \emph{pseudometric}. If $\r$ is only non-negative, symmetric, and $\r(x,x)=0$ for any $x\in X$, then we call such $d$ \emph{a distance function}, instead of metric or pseudometric. As a rule, if it is not ambiguous, we write $|xy|$ for $d(x,y)$.

In what follows all metric spaces are endowed with the corresponding metric topology. We also use the following notations. By $\# X$ we denote the cardinality of a set $X$. Let $X$ be a metric space. The closure of a subset  $A\ss X$  is denoted by $\bA$. For an arbitrary nonempty subset $A\ss X$ and point $x\in X$ put $|xA|=|Ax|=\inf\big\{|ax|: a\in A\big\}$. Further, for $r\ge 0$ put
$$
B_r(x)=\big\{y\in X: |xy|\le r\big\}, \quad\text{and}\quad U_r(x)=\big\{y\in X: |xy|< r\big\},
$$
and
$$
B_r(A)=\big\{y\in X: |Ay|\le r\big\}, \quad\text{and}\quad U_r(A)=\big\{y\in X: |Ay|< r\big\}.
$$

\subsection{Hausdorff distance}
Recall the basic results concerning the Hausdorff distance. The details can be found in~\cite{BurBurIva}. For a set $X$, by $\cP_0(X)$ we denote the collection of all nonempty subsets of $X$. Let $X$ be a metric space. For any $A,B\in\cP_0(X)$ we put
$$
d_H(A,B)=\inf\bigl\{r\in[0,\infty]:A\ss B_r(B)\ \text{and}\ B_r(A)\sp B\bigr\}.
$$
It is easy to see that $d_H$ is non-negative, symmetric, and $d_H(A,A)=0$ for any nonempty $A\ss X$, thus, $d_H$ is a generalized distance on the family $\cP_0(X)$ of all nonempty subsets of a metric space $X$, moreover, it is a generalized pseudometric on $\cP_0(X)$, i.e., it satisfies the triangle inequality. The function $d_H$ is referred as \emph{Hausdorff distance}.

Further, by $\cH(X)\ss\cP_0(X)$ we denote the set of all nonempty closed bounded subsets of a metric space $X$.  It is well-known that the Hausdorff distance $d_H$ is a metric on $\cH(X)$.

\subsection{Gromov--Hausdorff distance}
Let $X$ and $Y$ be metric spaces. A triple $(X',Y',Z)$ consisting of a metric space $Z$ and its two subsets $X'$ and $Y'$ which are isometric respectively to $X$ and $Y$ is be called a \emph{realization of the pair $(X,Y)$}. Put
$$
d_{GH}(X,Y)=\inf\bigl\{r\in\R:\text{there exists realization $(X',Y',Z)$ of $(X,Y)$ with $d_H(X',Y')\le r$}\bigr\}.
$$
The value $d_{GH}(X,Y)$ is evidently non-negative, symmetric, and $d_{GH}(X,X)=0$ for any metric space $X$. Thus, $d_{GH}$ is a generalized distance function on each set of metric spaces.

\begin{dfn}
The value $d_{GH}(X,Y)$ is called \emph{the Gromov--Hausdorff distance\/} between the metric spaces $X$ and $Y$.
\end{dfn}

It is well-known that the function $d_{GH}$ is a generalized pseudometric on every set of metric spaces. If the diameters of all spaces in the family are bounded by the same number, then $d_{GH}$ is a pseudometric. In general, $d_{GH}$ is not a metric, it may equal zero for distinct metric spaces. However, if we restrict ourselves to compact metric spaces considered up to an isometry, then $d_{GH}$ is a metric.

For specific calculations of the Gromov--Hausdorff distance, other equivalent definitions of this distance are useful.

Recall that \emph{a relation\/} between sets $X$ and $Y$ is defined as a subset of the Cartesian product $X\x Y$. Similarly to the case of mappings, for each $\s\in\cP_0(X\x Y)$ and for every $x\in X$ and $y\in Y$, there are defined the \emph{image\/} $\s(x):=\bigl\{y\in Y:(x,y)\in\s\bigr\}$ of any $x\in X$ and the \emph{pre-image\/} $\s^{-1}(y)=\bigl\{x\in X:(x,y)\in\s\bigr\}$ of any $y\in Y$. Also, for $A\ss X$ and $B\ss Y$ their \emph{image\/} and \emph{pre-image\/} are defined as the union of the images and pre-images of their elements, respectively.

A relation $R$ between $X$ and $Y$ is called a \emph{correspondence\/} if every $x\in X$ has a non-empty image, and every $y\in Y$ has a non-empty pre-image. Thus, the correspondence can be considered as a surjective multivalued mapping. By $\cR(X,Y)$ we denote the set of all correspondences between $X$ and $Y$.

If $X$ and $Y$ are metric spaces, then for each relation $\s\in\cP_0(X\x Y)$ its \emph{distortion $\dis\s$} as follows
$$
\dis\s=\sup\Bigl\{\bigl||xx'|-|yy'|\bigr|:(x,y),\,(x',y')\in\s\Bigr\}.
$$

The key well-known result on the relation between the correspondences and the Gromov--Hausdorff distance is the following equality
$$
d_{GH}(X,Y)=\frac12\inf\bigl\{\dis R:R\in\cR(X,Y)\bigr\},
$$
that is valid for arbitrary metric spaces $X$ and $Y$, see for example~\cite{BurBurIva}.

For arbitrary nonempty sets $X$ and $Y$, a correspondence $R\in\cR(X,Y)$ is called \emph{irreducible\/} if it is a minimal element of the set $\cR(X,Y)$ with respect to the order given by the inclusion relation. The set of all irreducible correspondences between $X$ and $Y$ is denoted by $\cR^0(X,Y)$. It can be shown~\cite{IvaTuz_corr} that for every $R\in\cR(X,Y)$ there exists $R^0\in\cR^0(X,Y)$ such that $R^0\ss R$. In particular, $\cR^0(X,Y)\ne\0$. Therefore, for any metric spaces $X$ and $Y$ the next equality holds
$$
d_{GH}(X,Y)=\frac12\inf\bigl\{\dis R\mid R\in\cR^0(X,Y)\bigr\}.
$$

Here we list several simple cases of exact calculation and estimate of the Gromov--Hausdorff distance.

\begin{examp}\label{examp:epsilon-net}
Let $Y$ be an arbitrary $\e$-net of a metric space $X$. Then $d_{GH}(X,Y)\le d_H(X,Y)\le\e$. Thus, every compact metric space is approximated (according to the Gromov-Hausdorff metric) with any accuracy by finite metric spaces.
\end{examp}

By $\D_1$ we denote a single-point metric space.

\begin{examp}\label{examp:GH_simple}
For any metric space $X$ we have
$$
d_{GH}(\D_1,X)=\frac12\diam X.
$$
\end{examp}

\begin{examp}\label{examp:dGHbelowEstimate}
Let $X$ and $Y$ be metric spaces, and the diameter of one of them is finite. Then
$$
d_{GH}(X,Y)\ge\frac12|\diam X-\diam Y|.
$$
\end{examp}

\begin{examp}\label{examp:GH-ineq-max-Diam-X-Y}
Let $X$ and $Y$ be some metric spaces, then
$$
d_{GH}(X,Y)\le\frac12\max\{\diam X,\diam Y\},
$$
in particular, if $X$ and $Y$ are bounded metric spaces, then $d_{GH}(X,Y)<\infty$.
\end{examp}

By \emph{simplex\/} we call a metric space, all whose non-zero distances equal to each other, i.e., a one-distance space. If $m$ is an arbitrary cardinal number, then by $\D_m$ we denote a simplex containing $m$ points and such that all its non-zero distances equal $1$. Thus, $\l\D_m$, $\l>0$, is a simplex whose non-zero distances equal $\l$. Also,  $0\, X$ coincides with $\D_1$ by definition. The following result holds, for a proof see~\cite{ITBorsFom}

\begin{prop}
Let $X$ be an arbitrary metric space, $m>\#X$ a cardinal number, and $\l\ge0$, then
$$
2d_{GH}(\l\D_m,X)=\max\{\l,\diam X-\l\}.
$$
\end{prop}

The case $2\le m\le \#X$ is rather more delicate, see details in~\cite{IvaTuzSimpDist8} and~\cite{GrigIvaTuzSimpDist}.

\subsection{Elements of Graph Theory}
By a \emph{graph\/} we mean a pair $G=(V,E)$ consisting of two sets $V$ and $E$ referred as the \emph{vertex set\/} and the \emph{edge set\/} of the graph $G$ , respectively; elements of $V$ are called \emph{vertices}, and the ones of $E$ are called \emph{edges\/} of the graph $G$. The set $E$ is a subset of the family of two-element subsets of $V$. If $V$ and $E$ are finite sets then the graph $G$ is called \emph{finite}.

It is convenient to use the following notations:
\begin{itemize}
\item If $\{v,w\}\in E$ is an edge of the graph $G$, then we write it just as $vw$ or $wv$; further one says that an edge $vw$ \emph{connects the vertices $v$ and $w$}, and that $v$ and $w$ are \emph{the vertices of the edge $vw$};
\item We write $V(G)$ and $E(G)$ for the vertex set and the edge set of a graph $G$ to underline which graph is under consideration.
\end{itemize}

Two vertices $v,w\in V(G)$ are called \emph{adjacent\/} if $vw\in E(G)$. Two different edges $e_1,e_2\in E(G)$ are called \emph{adjacent\/} if they have a common vertex, i.e., if $e_1\cap e_2\ne\0$. Each edge $vw\in E(V)$ and its vertex, i.e., $v$ or $w$, are said to be \emph{incident\/} to each other.  The cardinal number of edges incident to a vertex $v$ is called the \emph{degree of the vertex $v$} and is denoted by $\deg v$.

We also need some set-theoretical operations on graphs. They are usually defined in an intuitively clear way in terms of vertex and edge sets. For example, if $G=(V,E)$ is a graph, and $e$ is a two-element subset of $V$, then $G\cup e=\big(V,E\cup\{e\}\big)$; similarly for $e\in E$ put $G\sm e=\big(V,E\sm\{e\}\big)$.

The concept of hypergraph generalizes naturally the one of graph. Namely, a \emph{hypergraph $H=(V,E)$} is a pair, consisting of a \emph{vertex set} $V$ and an \emph{edge set $E$}, where $E$ is a family of nonempty subsets of $V$. As in the case of usual graphs, an element of $V$ is called a \emph{vertex}, and an element of $E$ is called an \emph{edge}.

\subsection{Generalized Borsuk Problem}\label{ssec:GBP}

Classical Borsuk Problem deals with partitions of subsets of Euclidean space into parts having smaller diameters. We generalize the Borsuk problem to arbitrary bounded metric spaces and partitions of arbitrary cardinality. Let $X$ be a bounded metric space, $m$ a cardinal number such that $2\le m\le\#X$, and $D=\{X_i\}_{i\in I}$ a partition of $X$ into $m$ nonempty subsets. We say that  $D$ is a partition of $X$ into subsets having \emph{strictly smaller diameters}, if there exists $\e>0$ such that $\diam X_i\le\diam X-\e$ for all $i\in I$.

The \emph{Generalized Borsuk Problem\/}: Is it possible to partition a bounded metric space $X$ into a given, probably infinite, number of subsets, each of which has a strictly smaller diameter than $X$?

In~\cite{ITBorsFom} a solution to this Problem in terms of the Gromov--Hausdorff distance is given.

\begin{thm}\label{thm:Borsuk}
Let $X$ be an arbitrary bounded metric space and $m$ a cardinal number such that $2\le m\le\#X$. Choose an arbitrary number $0<\l<\diam X$, then $X$ can be partitioned into $m$ subsets having strictly smaller diameters if and only if $2d_{GH}(\l\D_m,X)<\diam X$. If not, then $2d_{GH}(\l\D_m,X)=\diam X$.
\end{thm}

\subsection{Clique Cover Number and Chromatic Number of a Simple Graph}

Recall that a subgraph of an arbitrary simple graph $G$ is called \emph{a clique}, if any its two vertices are connected by an edge, i.e., the clique is a subgraph which is a complete graph itself. Notice that each single-vertex subgraph is a single-vertex clique. For convenience, the vertex set of a clique is also referred as a \emph{clique}.

On the set of all cliques, an ordering with respect to inclusion is naturally defined, and hence, due to the above remarks, a family of maximal cliques is uniquely defined; this family forms \emph{a covering of the graph $G$} in the following sense: the union of all vertex sets of all maximal cliques coincides with the vertex set $V(G)$ of the graph $G$.

If one does not restrict himself by maximal cliques only, then, generally speaking, one can find other families of cliques covering the graph $G$. One of the classical problems of Graph Theory is to calculate the minimal possible number of cliques covering a given finite simple graph $G$. This number is referred as the \emph{clique cover number\/} and is often denoted by $\theta(G)$. It is easy to see that the value $\theta(G)$ also equals the least number of cliques whose vertex sets form a partition of $V(G)$.

Another popular problem is to find the least possible number of colors that is necessary to color the vertices of a simple finite (hyper-)graph $G$ without monochromatic edges. This number is denoted by $\g(G)$ and is referred as the \emph{chromatic number of the\/ \(hyper-\/\)graph $G$}.

For a simple graph $G$, by $G'$ we denote its \emph{complement graph}, i.e., the graph with the same vertex set and the complementary set of edges (two vertices of $G'$ are adjacent if and only if they are not adjacent in $G$). It is easy to see that $\theta(G)=\g(G')$ for any simple finite graph $G$, see for example~\cite{West}.

Let $G=(V,E)$ be an arbitrary finite graph. Fix two positive real numbers $a<b\le2a$ and define a metric on $V$ as follows: the distance between adjacent vertices equals $a$, and the distance between nonadjacent vertices equals $b$. Then a subset $V'\ss V$ has diameter $a$ if and only if $G(V')\ss G$ is a clique. This implies that the clique cover number of $G$ equals the least possible cardinality of partitions of the metric space $V$ into subsets of (strictly) smaller diameter. However, this number was calculated in Theorem~\ref{thm:Borsuk}. Thus, we get the following result.

\begin{cor}\label{cor:clique}
Let $G=(V,E)$ be an arbitrary finite graph. Fix two positive real numbers $a<b\le2a$ and define a metric on $V$ as follows\/\rom: the distance between adjacent vertices equals $a$, and the distance between nonadjacent ones equals $b$. Let $m$ be the greatest positive integer $k$ such that $2d_{GH}(a\D_k,V)=b$ \(in the case when there is no such $k$, we put $m=0$\). Then $\theta(G)=m+1$.
\end{cor}

Because of the duality between clique cover and chromatic numbers, we get the following dual result.

\begin{cor}\label{cor:chrom}
Let $G=(V,E)$ be an arbitrary finite graph. Fix two positive real numbers $a<b\le2a$ and define a metric on $V$ as follows\/\rom: the distance between adjacent vertices equals $b$, and the distance between nonadjacent ones equals $a$. Let $m$ be the greatest positive integer $k$ such that $2d_{GH}(a\D_k,V)=b$ \(in the case when there is no such $k$, we put $m=0$\). Then $\g(G)=m+1$.
\end{cor}

\begin{cor}\label{cor:k_color}
Let $G=(V,E)$ be an arbitrary finite graph. Fix two positive real numbers $a<b\le2a$ and define a metric on $V$ as follows\/\rom: the distance between adjacent vertices equals $b$, and the distance between nonadjacent ones equals $a$. If $2d_{GH}(a\D_k,V)=b$, then $\g(G)>k$.
\end{cor}

\subsection{Embeddings into Rectilinear Spaces}
Recall that the rectilinear $m$-dimensional space $\R^m_1$ is the real $m$-dimensional space $\R^m$ endowed with the distance function generated by the norm $\|y\|_1=\sum_i|y_i|$ for $y=(y_1,\ldots,y_m)\in\R^m$. A pseudometric space $(X,d)$ is said to be \emph{$m$-embeddable into rectilinear space\/} if there exists an isometrical embedding $f\:X\to\R^m_1$. The latter means that $d(a,b)=\big\|f(a)-f(b)\big\|_1$ for all $a,\,b\in X$. As we have already mentioned above, the $m$-embeddability problem of arbitrary pseudometric space is equivalent to the one for its finite subspaces (see a proof in~\cite{DezaLaurent}).

To describe finite pseudometric spaces embeddable into rectilinear spaces we need to recall the concepts of cut pseudometrics and cut cones. Let $X$ be a finite set consisting of $n$ elements. Without loss of generality, we put $X=\{1,\ldots,n\}$. For a proper subset $S\subset X$, the pair $\{S,S'\}$, where $S'$ stands for $X\sm S$, is called a \emph{cut}. For a cut $c=\{S,S'\}$, define a pseudometric $\delta_c$ on $X$ as follows: $\delta_c(i,j)=1$ if $\#\big(S\cap\{i,j\}\big)=1$, and $\delta_c(i,j)=0$ otherwise. The pseudometric $\delta_c$ is referred as the \emph{cut metric\/} corresponding to the cut $c=\{S,S'\}$.

Recall that each pseudometric $d$ on $X=\{1,\ldots,n\}$ defines the square symmetric matrix $\big(d(i,j)\big)$ with $d(i,i)=0$ that can be given by the vector $\big(d(1,2),\ldots,d(n-1,n)\big)$ in $\R^N$, $N=n(n-1)/2$. The set of all such vectors is a convex cone in $\R^N$ which is called the \emph{metric cone}. Let $\cC$ be a family of cuts of $X$, and $\l\:\cC\to\R$ be a mapping such that $\l_c:=\l(c)>0$ for all $c\in\cC$. The pseudometric
$$
d(\cC,\l)=\sum_{c\in\cC}\l_c\dl_c
$$
is called a \emph{cut metric\/} corresponding to the family $\cC$ of cuts. Consider all such metrics $d(\cC,\l)$ over all cuts $\cC$ and all mappings $\l$. They form another cone in $\R^N$ that is referred as the \emph{cut cone\/} and is denoted by $\cut_n$.  The following criterion is obtained in~\cite{Assouad}.

\begin{ass}\label{ass:emb_cut}
A finite metric space $(X,d)$ is isometrically embeddable in a rectilinear space if and only if $d\in \cut_{\#X}$.
\end{ass}

\begin{rk}
The same criterion works for $L_1$-embeddability of finite metric spaces, i.e., a finite metric space $(X,d)$ is isometrically embeddable in an $L_1$-space if and only if $d\in \cut_{\#X}$ (and if and only if it is embeddable in a rectilinear space).
\end{rk}

\begin{rk}
The condition of Assertion~\ref{ass:emb_cut} seems simple, but it is rather difficult to verify. Moreover, it is proved that the problem is $\NP$-complete.
\end{rk}

If the answer is positive, that is, if a metric space $(X,d)$ turns out to be embeddable, then the question on the minimal admissible dimension of the rectilinear space arises naturally. This minimal dimension is referred as \emph{$\ell_1$-dimension\/} of the space $(X,d)$.

To state the coloring criterion found in~\cite{BCL} we need to construct a special hypergraph associated with a cut system $\cC$ of a set $X$. Two cuts $\{A,A'\}$ and $\{B,B'\}$ are said to be \emph{incompatible\/} if all four intersections $A\cap B$, $A\cap B'$, $A'\cap B$, and $A'\cap B'$ are non-empty. Three cuts $\{A,A'\}$, $\{B,B'\}$, and $\{C,C'\}$ are said to form an \emph{asteroid triplet\/} if one can choose one set $\tilde A\in\{A,A'\}$,  $\tilde B\in\{B,B'\}$, and $\tilde C\in\{C,C'\}$ from each cut in such a way that the three resulting sets $\tilde A$, $\tilde B$, and $\tilde C$ are pairwise disjoint. The hypergraph with vertex set $\cC$ whose edges are all incompatible pairs and all asteroid triplets is called \emph{nesting hypergraph of $\cC$} and is denoted by $\G(\cC)$.

One says that a hypergraph $G=(V,E)$ is \emph{$m$-colorable} if there exists a coloring with $m$ colors without monochromatic edges. In~\cite{BCL} the following result is proved.

\begin{ass}\label{ass:chrom-emb}
Let $(X,d)$ be a finite metric space, and
$$
d=\sum_{c\in\cC}\l_c\dl_c, \qquad \l_c>0,
$$
where $\cC$ is a cut family on $X$. Then $(X,d)$ is embeddable in $m$-dimensional rectilinear space if and only if the corresponding nesting hypergraph $\G(\cC)$ is $m$-colorable. In particular, the chromatic number of $\G(\cC)$ equals the $\ell_1$-dimension of $(X,d)$.
\end{ass}

\section{Embeddings an Gromov--Hausdorff Distance}
\markright{\thesection.~Embeddings an Gromov--Hausdorff Distance}

The main result of the paper is based on two results: Assertion~\ref{ass:chrom-emb} and Corollary~\ref{cor:chrom}.

Let $\cC$ be an arbitrary cut family of a finite set $X$, and $\G(\cC)$ the corresponding nesting hypergraph. Notice that for any its edge $\{a,b,c\}$ corresponding to an asteroid triplet, none of the pairs $\{a,b\}$, $\{b,c\}$, and $\{a,c\}$ forms an edge of $\G(\cC)$. Construct a family of simple graphs with vertex set $\cC$ as follows: for each asteroid triplet $\{a,b,c\}$ choose one of the pairs $\{a,b\}$, $\{b,c\}$, $\{a,c\}$, add it to the edge set and delete the triplet. Notice that such a pair could belong to several asteroid triplets, but if it is chosen several times, we add it to the edge set only once to avoid multiple edges. As a result, we get at most $3^k$ simple graphs $\{G_i\}$, where $k$ stands for the number of the asteroid triplets in the nesting hypergraph. We put $\cG(\cC)=\{G_i\}$.

\begin{rk}
If the nesting hypergraph does not contain any asteroid triplet, then it is a simple graph itself and the family $\cG(\cC)$ consists of a single element $\G(\cC)$.
\end{rk}

\begin{lem}\label{lem:hyp_simp}
Let $\cC$ be an arbitrary cut family of a finite set $X$, and $\G(\cC)$ corresponding nesting hypergraph, and let $\cG(\cC)=\{G_i\}$ be the family of simple graphs constructed above. Then $\G(\cC)$ is $m$-colorable if and only if the family $\cG(\cC)$ contains an $m$-colorable simple graph.
\end{lem}

\begin{proof}
Let $\G(\cC)$ be $m$-colorable, and fix some its $m$-coloring $\xi$ without monochromatic edges. Then each asteroid triplet contains at least two vertices of different color. Construct simple graph $G$ deleting each asteroid triplet and adding an edge connecting these two vertices. If vertices of such edge belong to several asteroid triplets, then we add this edge at most once. The resulting simple graph belongs to $\cG(\cC)$ and $\xi$ is its coloring without  monochromatic edges.

Conversely, let $G\in\cG(\cC)$ and $\xi$ be some its coloring without monochromatic edges. Then $\xi$ is also a coloring of the nesting hypergraph $\G(\cC)$ that does not contain monochromatic edges. Indeed, each asteroid triplet of $\G(\cC)$ contains an edge of $G$ and hence, contains at least two vertices of different color.
\end{proof}

\begin{cor}\label{cor:emd_simp_graph}
Let $(X,d)$ be a finite metric space, and
$$
d=\sum_{c\in\cC}\l_c\dl_c, \qquad \l_c>0,
$$
where $\cC$ is a cut family on $X$. Then $(X,d)$ is embeddable in $m$-dimensional rectilinear space if and only if the corresponding family  $\cG(\cC)$ of simple graphs contains an $m$-colorable graph. The $\ell_1$-dimension of $(X,d)$ equals $\min_G\g(G)$, where minimum is taken over all $G\in\cG(\cC)$.
\end{cor}

Now, fix two positive real numbers $0<a<b\le2a$, and for each graph $G$ from $\cG(\cC)$ construct a $2$-distant metric $d_G$ on $\cC$  as follows: put $d_G(s,t)=b$ if and only if $s$ and $t$ are adjacent, and $d_G(s,t)=a$ otherwise. By $\cC_G$ we denote the resulting metric space $(\cC,d_G)$.

\begin{thm}\label{thm:emb_GH}
Let $(X,d)$ be a finite metric space, and
$$
d=\sum_{c\in\cC}\l_c\dl_c, \qquad \l_c>0,
$$
where $\cC$ is a cut family on $X$. Then the $\ell_1$-dimension of $(X,d)$ equals $m$, if and only if $m$ is the least positive integer such that
$$
\min_G 2d_{GH}(a\D_m,\cC_G)<b,
$$
where minimum is taken over the family $\cG(\cC)$, and $a\D_m$ stands for the single-distance metric space with nonzero distance $a$.
\end{thm}

\begin{proof}
It is easy to see that $2d_{GH}(a\D_m,\cC_G)\le b$ for all positive integer $m$, and any $G\in\cG(\cC)$. Moreover, if $m\ge \#\cC$, then for any $G\in\cG(\cC)$ the distortion of any correspondence between $a\D_m$ and $\cC_G$, such that preimages of different cuts do not intersect each other is less than $b$. Therefore the set of positive integers such that
$$
\min_G 2d_{GH}(a\D_m,\cC_G)<b
$$
is not empty.

Let $m$ be the least positive integer such that
$$
\min_G 2d_{GH}(a\D_m,\cC_G)<b, \qquad G\in\cG(\cC).
$$
Let us start with the case $m=1$. Recall that  $2d_{GH}(a\D_1,\cC_G)=\diam \cC_G$, see Example~\ref{examp:GH_simple}, therefore in this case all the distances in the $2$-distance space $\cC_G$ equal $a$, and hence the graph $G$ is empty (i.e., it has no edges). So, the nesting graph $\G(\cC)$ is also empty and can be colored in one color. Therefore, $(X,d)$ is embeddable in the straight line $\R^1=\R^1_1$, and hence $\ell_1$-dimension of $(X,d)$ equals $1$.

Now, let $m\ge 2$. Due to assumptions, there exists $G_0\in\cG(\cC)$ such that $2d_{GH}(a\D_m,\cC_{G_0})<b$, and  $2d_{GH}(a\D_k,\cC_G)=b$ for all $k$, $1\le k<m$, and all $G\in\cG(\cC)$. That is, $(m-1)$ is the greatest positive integer such that $2d_{GH}(a\D_k,\cC_{G_0})=b$, and hence, due to Corollary~\ref{cor:chrom}, $\g(G_0)=m$, and $(X,d)$ is embeddable in $m$-dimensional rectilinear space in accordance with Corollary~\ref{cor:emd_simp_graph}.

Further, if $(X,d)$ is embeddable in $k$-dimensional rectilinear space, then, due to Corollary~\ref{cor:emd_simp_graph}, the family $\cG(\cC)$ contains a $k$-colorable graph $G$. Consider the corresponding $2$-distant metric space $\cC_G=(\cC,d_G)$.  In accordance with Corollary~\ref{cor:k_color}, $2d_{GH}(a\D_k,\cC_G)<b$, and hence $\min_G 2d_{GH}(a\D_k,\cC_G)<b$. Therefore $k\ge m$, due to assumptions, and so $m$ is the $\ell_1$-dimension of $(X,d)$.

Conversely, let $\ell_1$-dimension of $(X,d)$ equals $m$. As it is already shown above, the latter implies that $\min_G 2d_{GH}(a\D_m,\cC_G)<b$. If there exists $k<m$ such that $\min_G2d_{GH}(a\D_m,\cC_G)<b$, then, due to the direct statement, $(X,d)$ is embeddable in $k$-dimensional rectilinear space, and $\ell_1$-dimension of $(X,d)$ is less than $m$, a contradiction. Theorem is proved.
\end{proof}

\markright{References}

{\bf Affiliations}
\begin{description}
\item[Alexander Ivanov:] Lomonosov Moscow State University, Faculty of Mechanics and Mathematics, Leninskie Gory 1, Moscow, 119991, Russia; Bauman Moscow Technical State University, 2-aya Baumaskaya ul., d.~5, s.~1, Moscow, 105005, Russia.

e-mail: aoiva@mech.math.msu.su; alexandr.o.ivanov@gmail.com

\item[Alexey Tuzhilin:] Lomonosov Moscow State University, Faculty of Mechanics and Mathematics, Leninskie Gory 1, Moscow, 119991, Russia.

e-mail: tuz@mech.math.msu.su
\end{description}

\end{document}